\documentclass[12pt]{amsart}
\usepackage[utf8]{inputenc}
\usepackage[margin=1in]{geometry}
\usepackage{enumerate}
\usepackage{amsmath}
\usepackage{amssymb}
\usepackage{amsthm}
\usepackage{amsfonts}
\usepackage{xcolor}
\usepackage{comment}

\usepackage{hyperref}
\usepackage{url}

\usepackage{cleveref}
\usepackage{empheq}
\usepackage{graphicx}
\usepackage{tikz}
\usetikzlibrary{calc}
\usepackage{cite}
\hypersetup{
	colorlinks   = true,    
	urlcolor     = red,    
	linkcolor    = blue,    
	citecolor    = red      
}
\usepackage{soul}
\setul{0.2ex}{}
\usepackage{tikz-cd}

\newcommand{\Ima}{\text{Im}}

\newcommand{\Supp}{\text{Supp}}

\newcommand{\GL}{\text{GL}}
\newcommand{\SL}{\text{SL}}
\newcommand{\R}{\mathbb{R}}
\renewcommand{\P}{\mathbb{P}}

\newcommand{\N}{\mathbb{N}}

\newcommand{\F}{\mathbb{F}}
\newcommand{\Z}{\mathbb{Z}}
\newcommand{\C}{\mathbb{C}}
\newcommand{\Q}{\mathbb{Q}}

\newcommand{\Tr}{\mbox{Tr}}
\newcommand{\Prob}{\mbox{Prob\,}}

\newcommand{\eps}{\varepsilon}

\theoremstyle{plain}

\newtheorem{theo}{Theorem}
\newtheorem{coro}{Corollary}
\newtheorem{lemm}{Lemma}
\newtheorem{prop}{Proposition}

\theoremstyle{definition}
\newtheorem{defin}{Definition}
\newtheorem{rem}{Remark}
\newtheorem{exem}{Example}

\crefname{theo}{theorem}{Theorems}

\title{The one percent property for free groups}

\author{Emmanuel Breuillard}
\address{E. Breuillard, Mathematical Institute,
University of Oxford,
Andrew Wiles Building,
Woodstock Road,
Oxford OX2 6GG}
\email{emmanuel.breuillard@maths.ox.ac.uk}

\author{Ferdinand Jacobé de Naurois}
\address{F. Jacobé de Naurois, Mathematical Institute,
University of Oxford,
Andrew Wiles Building,
Woodstock Road,
Oxford OX2 6GG}
\email{ferdinand.jacobedenaurois@maths.ox.ac.uk}

\begin{document}

	\maketitle
	\begin{abstract}
		We prove that for every $\eps >0$, there is $\delta >0$ such that for every probability measure $\mu$ on a free group $\F$, if the operator norm of the action of $\mu$ on $L^2\F$ by convolution exceeds $\eps$, then there is a coset of an abelian subgroup of $\F$ with $\mu$-mass at least $\delta$.
	\end{abstract}

	\bigskip
	\bigskip
	
	Consider a discrete group $\Gamma$, with the Hilbert space $L^2\Gamma$ of $L^2$ functions from $\Gamma$ to $\R$, endowed with the standard scalar product:
	$$\langle f,g\rangle  = \sum_{x\in \Gamma}f(x)g(x).$$
	 Let $(\lambda, L^2\Gamma)$ be the left regular representation of $\Gamma$, that is the unitary representation, where $\lambda(g) : L^2\Gamma\to L^2\Gamma$ is the unitary operator defined by $\lambda(g)(f) : x\mapsto f(g^{-1}x)$. For a probability measure $\mu$ on $\Gamma$ and  every $f\in L^2\Gamma$, set:
	\begin{align*}
	\mu * f : \Gamma &\to \R\\
	x&\mapsto \sum_{g\in \Gamma}\mu(g)f(g^{-1}x).
	\end{align*}
	The function $\mu*f$ is still in $L^2\Gamma$, and the map $\lambda(\mu) : L^2\Gamma\to L^2\Gamma$ sending $f$ to $\mu*f$ is a bounded operator, whose norm we denote by $\|\lambda(\mu)\|$. Note that $\|\lambda(\mu)\|\leq 1$. When $\mu$ is symmetric (that is $\mu(g^{-1})=\mu(g)$ for each $g \in \Gamma$), the operator $\lambda(\mu)$ is self-adjoint. In this case, its norm coincides with its spectral radius.
	
	Kesten \cite{kesten} showed that for symmetric $\mu$, \begin{equation}\label{kest} \lim_{n \to +\infty} \mu^{2n}(1)^{1/2n} =  \sup_{n \ge 1} \mu^{n}(1)^{1/n} = \|\lambda(\mu)\|\end{equation} where $1$ is the identity of the group $\Gamma$, and moreover $\|\lambda(\mu)\|=1$ if and only if the support of $\mu$ generates an amenable subgroup of $\Gamma$. The symmetry condition cannot be omitted: e.g. in the free group on two generators $a$ and $b$ the measure $\mu = \frac{1}{2}(\delta_a+\delta_b)$ has norm $1$.  One may however often reduce to the symmetric case by considering the symmetrized measure $\widetilde{\mu}:=\mu*\mu^*$, where $\mu^*(g):=\mu(g^{-1})$, since $\|\lambda(\widetilde{\mu})\|=\|\lambda(\mu)\|^2$.

	In this text, we are mostly interested in the other regime, that is when $\|\lambda(\mu)\|$ approaches zero. We begin with the following simple observation:
    
	\begin{prop}		
		For every probability measure $\mu$ on $\Gamma$, every amenable subgroup $A$ in $\Gamma$ and every $g\in \Gamma$, $\|\lambda(\mu)\|\geq \mu(gA)$.
	\end{prop} 
	
	We include the proof here, since it is very short:
	
	\begin{proof}
		 Replacing $\mu$ by $\delta_{g^{-1}}*\mu$ does not change the operator norm, so we may assume that $g = 1$. Let $\delta >0$ and $\nu := \frac{\mu_{|A}}{\mu(A)}$. We pick a non-zero $\xi\in L^2A\subseteq L^2\Gamma$ such that $\|\nu * \xi\|\geq( 1-\delta)\|\xi\|$. Write $\xi'(x)  =|\xi(x)|$ for every $x\in \Gamma$. Then $\|\xi'\| = \|\xi\|$ and $\|\nu*\xi'\|\geq \|\nu*\xi\|$, hence we may assume that $\xi\ge 0$. It follows that $\|\mu*\xi\|\geq \|\mu_{|A}*\xi\| \geq \mu(A)(1-\delta)\|\xi\|$. Hence upon letting $\delta\to 0$ we get $\|\lambda(\mu)\|\geq \mu(A)$ as desired.
		\end{proof}
	
	Our goal is to prove a converse to the above proposition for a certain class of groups, including free groups. We introduce the following terminology:
	\begin{defin}[One percent property] \label{defin one percent}
		Say a discrete group $\Gamma$ has the \emph{one percent property} if and only if for every $\varepsilon >0$, there is $\delta>0$ such that the following holds:		
			 For every probability measure $\mu$ on $\Gamma$ with $\|\lambda(\mu)\| > \varepsilon$, there is an amenable subgroup $A$ of $\Gamma$ and $g\in \Gamma$, such that $\mu(gA)\geq \delta$.

	\end{defin}
	Of course, in the above definition, $\delta$ is allowed to depend on $\varepsilon$ but not on the specific probability measure $\mu$. Our main result is as follows:
	\begin{theo}\label{main theorem}
		Free groups have the one percent property.
	\end{theo}

Since free groups of arbitrary countable rank arise as subgroups of any nonabelian free group and since the one-percent property is clearly stable under passing to a subgroup,  we see that the same modulus $\delta=\delta(\varepsilon)>0$ works for all free groups simultaneously. We discuss an explicit determination of a modulus  in \Cref{explicit} below.

\bigskip
 Although we will not use this fact, we note and it is not hard to see that the one percent property is stable under the direct product of two groups (but not of an infinite number of groups, see \Cref{resfree}). It is also easily seen that amenable-by-one-percent groups are one-percent (with the same modulus). However, by contrast with  C*-simplicity \cite{bkko} or uniform exactness \cite{ozawa2026uniform}, the example of Baumslag-Solitar groups (see \Cref{BS}), which are free-by-solvable, shows that the one percent property is not stable under group extension.

We describe now another way in which the property can be inherited.  Given two groups $\Gamma$ and $L$, recall that $L$ is said to be fully residually $\Gamma$ if for every finite subset $S$ of $L$, there is a group homomorphism $\varphi : L\to\Gamma$, which is injective on $S$.
 
	\begin{prop}\label{frf}		Let $\Gamma$ be a discrete group with the one percent property. Assume that all amenable subgroups of $\Gamma$ are abelian.	If $L$ is fully residually $\Gamma$, then $L$ too has the one percent property (with the same modulus).
	\end{prop}
	In view of the above proposition, \Cref{main theorem} implies:
	\begin{coro}\label{frfg}
		Fully residually free groups have the one percent property.
	\end{coro}
Note that this class of groups includes fundamental groups of orientable surfaces.
	
	\bigskip
	
	That the one-percent property for free groups might hold was first investigated in Cagri Sert's doctoral thesis \cite[Chapter 7]{sert}, where partial results were obtained in this direction. It was  stated again as an open problem in \cite[\S 1.5.9]{becker-breuillard}. We believe it likely that not only free groups but all hyperbolic groups and all finitely generated subgroups of matrices over any field have this property. In fact we believe that $\GL_d(\C)$, viewed as a discrete group, has the one-percent property. A weaker property, namely the ninety-nine percent property (see \Cref{ninety}), which is closely related to uniform non-amenability, was shown to hold in this generality in \cite[Theorem 7.6]{becker-breuillard}. 
	
	Almost two decades ago, the first author realized that the one-percent property for a Zariski-dense subgroup in a centerfree simple algebraic group would imply their C*-simplicity via a simple route  - the so-called uniform Dixmier property -- as in Bekka, Cowling and de la Harpe \cite[Proposition 3]{bch}. However, the one-percent property turned out to be harder than C*-simplicity itself, which was later established in \cite{bkko}; see also the recent \cite{vigdorovich, hartman-kalantar} for the uniform Dixmier property.

	Our proof of \Cref{main theorem} makes use of Ozawa's recent proof that free groups are uniformly exact \cite{ozawa2026uniform} along with a new kind of Tits alternative, which is of independent interest and characterizes algebraic subvarieties of $\SL_2$ admitting tuples of free group generators:

\begin{prop}[Tits alternative for varieties]\label{alternative} Suppose $\mathcal{C}$ is an irreducible algebraic subvariety of $\SL_2(\C)$. Let $s \ge 2$ be an integer. Then one of the following holds:
\begin{enumerate}[(i)]
\item there is an integer $k \ge 1$ and $v_1,v_2 \in \C^2\setminus\{0\}$ such that $x^kv_1 \in \C v_2$ for all $x \in \mathcal{C}$,
\item there are $x_1,\ldots,x_s \in \mathcal{C}$ that generate a nonabelian free subgroup of rank $s$ in $\SL_2(\C)$. 
\end{enumerate}
\end{prop}
	
	The two conditions are mutually exclusive when $s \ge 3$. This alternative has the following consequence for subsets of the free group. Given positive integers $s$ and $L$ we say that a subset $\Omega$ of a group $\Gamma$ is \emph{$(s,L)$-related} if every $x_1,\ldots,x_s \in \Omega$ admit a relation of length at most $L$, i.e. a nontrivial word of length at most $L$ such that $w(x_1,\ldots,x_s)=1$. We shall show the following, where $\F$ is a nonabelian free group:
	
	\begin{prop}\label{uniform related ultrapower}
		Given $s,L\geq 1$, there is $c(s,L)\leq 2(2s)^{3L}L^{3}$ such that every $(s,L)$-related subset of $\F$ is contained in a union of at most $c(s,L)$ cosets of cyclic subgroups in $\F$.
	\end{prop}

	\bigskip
	
	\noindent{{\bf Notational conventions:}} We denote by $\N=\{0,1,\ldots\}$ the set of natural numbers, $\Z$ the relative integers and by $ [\![a,b]\!]$ the interval between $a$ and $b$ in $\Z$. A nonabelian free group is denoted by $\F$ and $[x,y]=xyx^{-1}y^{-1}$ denotes the commutator of two group elements.

	\bigskip

	\noindent{{\bf Statement of AI use:}} The main results of this manuscript (the material up to \Cref{secproof main} included) were finalized in early July 2026 and discovered without any use of AI. The manuscript was then written without AI help, except for proofreading and looking up references. In particular, reference \cite{lyndon1962equation} was suggested by ChatGPT 5.5. The last section (\Cref{further}) however contains complementary remarks that have been obtained in September 2026 with some AI help.
	\bigskip
	
	\noindent{{\bf Acknowledgments:}} The first author acknowledges support from UKRI grant no. 1017.  The second author acknowledges support from the Clarendon Fund and Keble College, University of Oxford. For the purpose of open access, the authors have applied a CC BY public copyright licence to any author accepted manuscript arising from this submission.

	\section{Preliminaries}

				Let $\F$ be a non abelian free group. The following is classical:		\begin{lemm}
			The free group $\F$ satisfies the two following conditions.
			
			$(i)$ Let $u\in \F$ be a non-trivial element. The centralizer of $u$, denoted by $C_\F(u)$, is cyclic.
			
			$(ii)$ For every $u\in\F$ such that $u\neq 1$ and all $g\in \F$, if $[u,g^{-1}ug] = 1$ then $[g,u] = 1$.
		\end{lemm}
		\begin{proof}
			We refer to \cite[Chapter I, Theorem 2.19]{lyndon1977combinatorial} for $(i)$. 
			For $(ii)$ pick $u\in\F\backslash \{1\}$ and $g\in \F$ such that $[u,g^{-1}ug] = 1$. Denote by $v$ a generator of $C_\F(u)$, so that $g^{-1}ug = v^k$ for some $k\in\Z$ and $u = v^m$ for some $m\in\Z$. Thus we have $g^{-1}v^mg = v^k$. The cyclically reduced length $l(w)$ of a word $w\in\F$, defined by the smallest word length of a conjugate of $w$, is invariant by conjugation hence $|m|l(v) = l(v^m) = l(v^k) = |k|l(v)$ so $k = \pm m$ and $g^{-1}ug = u^{\pm 1}$. If $g^{-1}ug = u$ we are done. Otherwise, $g^{-2}ug^2= u$ hence $g^2$ commutes with $u$, so $u,g\in C_\F(g^2)$ and either $g^2 = 1$ so $g = 1$, or $[g,u] = 1$. In both cases $[g,u] = 1$.
		\end{proof}
		Note that $(i)$ implies that $\F$ is commutative transitive: that if $u$ commutes with $v\neq 1$ and $v$ with $w$, then $u$ commutes with $w$. 
		
		The \emph{abelian centralizers}, that is the  subgroups $C_\F(u)$ for $u\neq 1$, are maximal abelian subgroups. Every abelian subgroup of $\F$ is contained in such a subgroup. In fact:

		\begin{lemm}\label{solvable abelian}
			
			Every solvable subgroup of $\F$ is cyclic.
		\end{lemm}
		\begin{proof}
			Let $H$ be a solvable subgroup of $\F$ and pick $n\in\N$ minimal with $H^{(n)} =\{1\}$. Then $A = H^{(n-1)}$ is abelian non trivial. Pick $u\in A$ non-trivial. For every $g\in H$, $g^{-1}ug\in H^{(n-1)}$ hence by $(ii)$ above $[g,u] = 1$, so $H\subseteq C_{\F}(u)$.
		\end{proof}

\subsection{Spectral radius of free sets}

Kesten \cite{kesten} showed in his thesis that if $x_1,\ldots,x_s$ are elements of a discrete group $\Gamma$ and $\nu$ is the uniform measure on the $2s$ elements $x_1,x_1^{-1},\ldots,x_s,x_s^{-1}$, then $$\|\lambda(\nu)\|\ge \frac{\sqrt{2s-1}}{s}$$
with equality if and only if $\langle x_1,\ldots,x_s\rangle$ is free of rank $s$.
Furthermore it is known (cf. \cite[Theorem IV J]{akemann1976computing} and \cite[Theorem 9]{lehner}) that if $\mu$ denotes the uniform measure on $x_1,\ldots,x_s \in \Gamma$ and $s \ge 3$, then \begin{equation}\label{freeradius}\|\lambda(\mu)\| \ge b_s:=\frac{2\sqrt{s-1}}{s}\end{equation} with equality if and only if $x_1^{-1}x_2,\ldots,x_1^{-1}x_s$ generates a free subgroup of rank $s-1$.

		\subsection{Uniform exactness}
	Recall the following definition from \cite{ozawa2026uniform}. The definition we give is not the original one used in \cite{ozawa2026uniform}, however it is equivalent to it due to \cite[Theorem 30]{ozawa2026uniform}. 
	\begin{defin}[uniform exactness]

				A discrete group $\Gamma$ is said to be uniformly exact if the following property holds.
		There is a function $k : \N\times (0,1)\to \N$ such that for every $\eps > 0$ and every finite subset $F\subseteq \Gamma$, there is a finite subset $E\subseteq \Gamma$ with $|E|\leq k(|F|,\varepsilon)$ and a map $\eta : \Gamma\to \Prob \Gamma$, $x \mapsto \eta^x$ with
		$$\bigcup_{x\in \Gamma}\text{supp}(\eta^x)\subseteq E \qquad \text{and} \qquad \sum_{s\in F}\sup_{x\in \Gamma} \|\eta^{sx}-s\eta^x\|_1\leq \varepsilon.$$
	\end{defin}
	Here $\Prob \Gamma$ denotes the set of probability measures on $\Gamma$. A key consequence of uniform exactness is the following \emph{locality property of spectral radius}, see \cite[Proposition D]{ozawa2026uniform}.  If $\mu$ is a finitely supported  probability measure on $\Gamma$ and $\delta>0$, there is $M_{\delta,s}>0$ depending only on $\delta$ and the cardinality $s$ of the support of $\mu$, such that \begin{equation}\label{ue} \widetilde{\mu}^{n}(1)^{1/2n} \ge (1-\delta) \|\lambda(\mu)\|\end{equation} for all $n \ge M_{\delta,s}$.

	 Along with introducing uniform exactness in \cite{ozawa2026uniform} Ozawa proved that:
	\begin{theo}[Ozawa]\label{oz}
		Fully residually free groups are uniformly exact.
	\end{theo}
	We will only use this fact through the following consequence.
	
	\begin{coro}\label{small relation if big norm}
		Let $\Gamma$ be a  uniformly exact discrete group, and let $s\geq 2$ be a positive integer and $b_s=2s^{-1}(s-1)^{1/2}$. If $\mu$ is the uniform measure on $x_1,\ldots,x_s \in \Gamma$, then either $\|\lambda(\mu)\|\leq 2b_s$, or there is a word $w\in\F_s\backslash\{1\}$ of length $\leq 2M_{1/2,s}$ such that $w(x_1,\cdots,x_s) = 1$.
	\end{coro}
	\begin{proof} If $\mu$ denotes the uniform probability measure on $s$ (not necessarily distinct) elements of $\Gamma$ and if these elements do not have any relation of length at most $2n$, then $\widetilde \mu^n (1)=\widetilde \nu^n(1)$, where $\nu$ is the uniform measure on the $s$ free generators of a free group of rank $s$. In particular $\widetilde \mu^n(1)\leq \|\lambda(\nu)\|^{2n} = b_s^{2n}$ by \eqref{freeradius}.
	
	On the other hand, by Theorem \ref{oz} and the remark above it, $\widetilde \mu^n (1)^{1/2n} \ge \|\lambda(\mu)\|/2$ as soon as $n \ge M_{1/2,s}$. So either $\|\lambda(\mu)\| \leq 2b_s$, or there is a relation of length at most $2M_{1/2,s}$. 
	\end{proof}

	\section{A Tits alternative for subvarieties}
	In this section, we prove \Cref{alternative}. The main argument of this proof is contained in the following lemma. Recall that if $K$ is a field, then a Borel subgroup of $\SL_2(K)$ is by definition a subgroup of the form $\{g\in \SL_2(K)\; |\; gv\in Kv\}$ for some vector $v\in K^2\backslash\{0\}$.
	\begin{lemm}\label{tits}
		Let $K$ be an algebraically closed field of characteristic zero, and $\mathcal{C}$ an irreducible algebraic variety in $\SL_2(K)$. Assume that there is a non-trivial word in $s$ letters $w \in \F_s\backslash\{1\}$ such that $w(\mathcal{C}^s)$ lies in a coset of a Borel subgroup. Then there is an integer $k\geq 1$ such that $p_k(\mathcal{C}) := \{g^k \;|\; g\in \mathcal{C} \}$ lies in a coset of a Borel subgroup. 
	\end{lemm}
	\begin{proof}
		By a standard base change argument, it suffices to prove the result over $\C$. For a vector $v \in \C^2$ write $\langle v \rangle$ for the line $\C v$. We prove the lemma by induction on $s\geq 1$. The case $s = 1$ is obvious. Therefore, it suffices to show that if there is a non-trivial word $w$ in $s\geq 2$ letters such that $w(\mathcal{C}^{s})$ is contained in a coset of a Borel subgroup, then the same holds for $s-1$. If $\mathcal{C}$ is zero dimensional (hence just a point) there is nothing to prove so we assume that $\dim(\mathcal{C})\geq 1$. We may label the letters such that $w = x_s^{n_1}w_1x_s^{n_2}\cdots x_s^{n_r}w_r$ where the $w_i$ are words in the letters $x_1,\cdots,x_{s-1}$ and such that $w_i$ is non-trivial for all $i\leq r-1$.
		
		We pick two non-zero vectors $v_1,v_2\in \C^2$ such that $w(\mathcal{C}^s)v_1\subseteq \C v_2$. Since $\mathcal{C}$ is unbounded, we pick a sequence $(g_n)_{n\geq 1}$ of elements of $\mathcal{C}$ such that $\|g_n\| \xrightarrow[n\to+\infty]{} +\infty$. If $p_k(\mathcal{C})$ is zero dimensional for some $k\geq 1$ (as a constructible set, hence just a point because it is the image of an irreducible variety by a regular map) then the conclusion holds. So we assume it is of dimension at least $1$ for every $k$, hence unbounded. We can then arrange so that the sequence $\|g_n^{n_i}\|$ tends to $+\infty$ for all $i\in [\![1,r]\!]$. To see it, note that $\frac{1}{2}\Tr(g^k) = T_k(\frac{1}{2}\Tr(g))$ for all $g\in \SL_2(\C)$ where $T_k$ is the Chebychev polynomial of the first kind, so $g^k$ tends to infinity provided $\Tr(g)$ does. This justifies the existence of $(g_n)$ when $\Tr(g)$ is not constant on $\mathcal{C}$. If $\Tr(g)$ is constant on $\mathcal{C}$ then by Cayley Hamilton we get $g^k = c_1g+c_2$ for all $g\in \mathcal{C}$, for constants $c_1,c_2$ depending on $k$ but not on $g\in\mathcal{C}$. If $c_1 = 0$, then $p_k(\mathcal{C})$ is a point, and otherwise $g^k$ tends to infinity with $g$ as desired.
		
		Passing to a subsequence if necessary, we may assume that $\frac{g_n^{n_i}}{\|g_n^{n_i}\|}$ converges to a rank one matrix $d_i$ for every $i\in[\![1,r]\!]$. Now for each $n\in \N$ and each $\underline{x}\in\mathcal{C}^{s-1}$ there is a scalar $\lambda(n,\overline{x})$ such that:
		\begin{equation}\label{firsteq}g_n^{n_1}w_1(\overline{x})g_n^{n_2}\cdots g_n^{n_r}w_r(\overline{x})v_1 = \lambda(n,\overline{x})v_2\end{equation}
		Dividing by the product of the $\|g_n^{n_i}\|$ and letting $n\to+\infty$, we get a scalar $\lambda(\overline{x})$ such that:
		\begin{equation}\label{secondeq}d_1w_1(\overline{x})d_2\cdots d_r w_r(\overline{x})v_1 = \lambda(\overline{x})v_2\end{equation}
		Note that $\lambda(\overline{x})$ is finite and is a regular function of $\overline{x}\in\mathcal{C}^{s-1}$ because the left hand side is. In particular, by irreducibility either $\lambda$ is identically zero or generically non-zero.
		
		Suppose first that $\lambda$ vanishes identically on $\mathcal{C}^{s-1}$. Then by \eqref{secondeq} above we see that for each $\overline{x}\in\mathcal{C}^{s-1}$ there is $i\in [\![1,r]\!]$ such that $w_i(\overline{x})\Ima(d_{i+1})\subseteq \ker(d_i)$ where by $\Ima(d_{r+1})$ we mean $\langle v_1\rangle$. By irreducibility one of those identities holds identically over $\mathcal{C}^{s-1}$, which concludes the induction step unless $w_r  = 1$, a case we postpone for now.
		
		Suppose now that $\lambda(\overline{x})$ does not vanish identically. Then $w_1(\overline{x})d_2w_2(\overline{x})\cdots d_rw_r(\overline{x})v_1$ does not vanish identically either. But it is the limit of $\frac{\|g_n^{n_1}\|\lambda(n,\overline{x})}{P_n}g_n^{-n_1}v_2$ where $P_n$ is the product of the $\|g_n^{n_i}\|$. Now we know that $\frac{\lambda(n,\overline{x})}{P_n}$ converges to $\lambda(\overline{x})$, therefore $\|g_n^{n_1}\|g_n^{-n_1}v_2$ has to converge to some non-zero vector $v_3\in\C^2$, and:
		$$w_1(\overline{x})d_2\cdots d_rw_r(\overline{x})v_1= \lambda(\overline{x})v_3 $$
		for generic $\overline{x}\in\mathcal{C}^{s-1}$. Hence $w_1(\mathcal{C}^{s-1})\Ima(d_2)\subset\langle v_3\rangle$. This concludes the induction step in this case.
		
		It remains to treat the case where $w_r = 1$, that is when the word starts and ends with a non-trivial power of $x_s$. Then instead of substituting $g_n$ in $x_s$, we substitute it into $x_{s-1}$ provided it appears in $w$ (if it does not then the induction step follows directly). This leads to:
		$$w_1'(\overline{x})g_n^{m_1}\cdots g_n^{m_{r'-1}}w'_{r'}(\overline{x})v_1 = \mu(n,\overline{x})v_2$$
		where the $m_i$ are all non-zero integers, the $w_i'$ are all non-trivial words and $\mu(n,\overline{x})$ is a scalar. Choosing $g_n$ so that the $\|g_n^{m_i}\|$ tend to $+\infty$ as before, dividing by the product of their norms and extracting as before, we get:
		$$w_1'(\overline{x})d_1'\cdots d_{r'-1}'w'_{r'}(\overline{x})v_1 = \mu(\overline{x})v_2$$
		where the $d_i'$ are rank one matrices and $\mu(\overline{x})$ is a regular function on $\mathcal{C}^{s-1}$. If it does not vanish identically then it is non-zero generically then $w_1'(\overline{x})\Ima(d_1 ')\subseteq \langle v_2\rangle$ as before, proving the induction step. If it is identically zero then for every $\overline{x}\in \mathcal{C}^{s-1}$ there is $i\in [\![2,r']\!]$ such that $w_{i}'(\overline{x})\Ima(d_i')\subseteq \ker(d_{i-1}')$, where by $\Ima(d_{r'}')$ we mean $\langle v_1\rangle$. Again by irreducibility one of these identities has to be true generically, proving the induction step.
			\end{proof}

			\begin{rem}\label{kbound}The word $w$ is a product of successive powers of letters in the alphabet $\{x_1,\ldots,x_s\}$. Note that carefully reviewing the above argument, we see that $k$ can be taken to be the absolute value of one of these powers. In particular, $k\leq \ell(w)$, the word length of $w$. 
			\end{rem}

We may now conclude the proof of \Cref{alternative}.

\begin{proof}[Proof of \Cref{alternative}] If the first item does not hold, then $w(x_1,\ldots,x_s)=1$ cuts out a proper algebraic subvariety of $\mathcal{C}^s$ according to \Cref{tits} for every nontrivial word $w$ in $s$ letters. A countable union of proper algebraic subvarieties of $\mathcal{C}^s$ cannot cover all of $\mathcal{C}^s$ by irreducibility of $\mathcal{C}^s$ (e.g. by Baire category). The conclusion follows.
\end{proof}

\begin{rem}Note that when $s \ge 3$ the two conditions are mutually exclusive:  If the first item holds, then $x^{-k}y^k$ belongs to a Borel subgroup for all $x,y\in \mathcal{C}$. In particular, $x^{-k}y^k$ and $x^{-k}z^k$ can never generate a free subgroup for any $x,y,z \in \mathcal{C}$. However, if $\langle x,y,z \rangle$ is free of rank $3$, then $\langle x^{-k}y^k, x^{-k}z^k\rangle$ is free of rank $2$. When $s=2$ the conditions can intersect: we leave it to the reader to verify that for every $g \notin B$ one can find $x,y \in gB$ that play ping-pong on $\P^1(\C)$, so that $\langle x,y \rangle$ is free of rank $2$.
		\end{rem}


    
    \section{Related sets in free groups}\label{secrelated}
    The goal of this section is to prove \Cref{uniform related ultrapower}. Let $s, L$ be positive integers. 	We define the notion of $(s,L)$-related sets as follows.
		\begin{defin}
		We say that a subset $\Omega$ of a group $\Gamma$ is $(s,L)$-related if every $s$-tuple of elements of $\Omega$ satisfies a non-trivial relation of length at most $ L$.
	\end{defin}

	The following is  consequence of \Cref{alternative}.
	\begin{prop}\label{related subsets and pk}
		Let $\Omega$ be an $(s,L)$-related subset of $\F$. There is a decomposition $\Omega = \Omega_1\cup \cdots \cup\Omega_d$, with $d \leq d(s,L):=2(2s-1)^{3L}L^{3}$, such that for every $i\in [\![1,d]\!]$, there is $k_i \in [\![1,L]\!]$ such that $p_{k_i}(\Omega_i)$ is contained in a coset of an abelian subgroup of $\F$.
	\end{prop}
	\begin{proof}
		Consider $\F$ embedded into $\SL_2(\C)$. Pick $\Omega$ an $(s,L)$-related subset of $\F$, and $\mathcal{C}$ its Zariski closure. Note that $\mathcal{C}$ too is $(s,L)$-related. Write $\Omega_i = \Omega\cap \mathcal{C}_i$, where the $\mathcal{C}_i$ are the irreducible components of $\mathcal{C}$. By irreducibility of $\mathcal{C}_i^s$ there is a non-trivial word $w$ in $s$ letters with length at most $L$ such that $w(\mathcal{C}_i^s) = \{1\}$. Applying \Cref{alternative} there is $k_i \geq 1$ with $k_i \leq L$ such that $p_{k_i}(\Omega_i) \subset p_{k_i}(\mathcal{C}_i)$ is contained in a coset of a Borel subgroup of $\SL_2(\C)$. In particular the group generated by the $xy^{-1}$ for $x,y\in p_{k_i}(\Omega_i)$ is a solvable subgroup of $\F$, hence abelian by \Cref{solvable abelian}. So $p_{k_i}(\Omega_i)$ is contained in a coset of an abelian subgroup of $\F$, hence in a coset of a maximal abelian subgroup, i.e. a centralizer.
		
		We can bound $d$ as follows. Note that we may replace $\mathcal{C}$ by a maximal algebraic subvariety of $\SL_2$ such that $\mathcal{C}^s \subset W$, where $W:=\cup_{\ell(w)\leq L} w^{-1}(1)$. Then $d$  is bounded by the number of irreducible components of $\mathcal{C}$ and hence by its degree (defined as the sum of the degrees of its components).  Each $w^{-1}$ is cut by polynomials in the matrix entries of degree at most $L$, so $W$ is cut by polynomials of degree at most $D:=(2s-1)^L L$. It is easy to check that the algebraic set $V$ cut out by the polynomials of degree at most $D$ that vanish on $\mathcal{C}$ will also satisfy $V^s \subset W$, so that $V=\mathcal{C}$ by maximality of $\mathcal{C}$.  It follows from the (relative) Bezout inequality (cf. \cite[Prop. 2.3]{heintz-schnorr}) that for closed subvarieties $X,Y$ of $\C^n$ we have $\deg(X \cap Y) \leq \deg(Y) D^{\dim Y}$ if $X$ is cut out by polynomials of degree at most $D$. Thus $d \leq \deg \mathcal{C} \leq \deg(\SL_2) D^3 = 2D^3.$
	\end{proof}


	Recall the following theorem from \cite{lyndon1962equation}.
	\begin{theo}[Lyndon-Schützenberger]\label{lyndon schutzenberg}
		Let $n,m,p\in \N$ all greater than or equal to $2$. If $a,b,c\in \F$ satisfy $a^n = b^mc^p$, then there is an element $u\in \F$ such that $a,b,c\in \langle u \rangle$.
	\end{theo}
	See also \cite{baumslag} for a generalisation and \cite[Theorem 1.1]{louder} for a recent geometric proof. The corollary of importance to us is the following.
	\begin{coro}\label{powers}
		Let $g,u\in \F$ be two non-commuting elements and let $k\geq 2$ be an integer. There can be at most two  elements $x\in \F$ such that $x^k \in gC_\F(u)$.
	\end{coro}
	\begin{proof}
		Assume there are three distinct ones: $x_1,x_2,x_3$. Then the $g^{-1}x_i^k$ all lie in the cyclic subgroup $C_\F(u)$, hence there is $v\in C_\F(u)$ and $n_1,n_2,n_3\in \Z$ such that $x_i^k = gv^{n_i}$ for every $i$, and $v\neq 1$, because the $x_i^k$ are distinct (indeed: if $x^k=y^k$ then $x$ commutes with $x^k=y^k$ hence with $y$, which reduces this statement to the case where $x$ and $y$ commute). Hence the $n_i$ are distinct. Without loss of generality we may assume $n_1 < n_2 < n_3$.  Then $d := n_3 - n_1 \geq 2$, and  $(x_1^{-1})^k x_3^k = v^{-n_1+n_3} = v^d$. Therefore, by the above \Cref{lyndon schutzenberg}, we have that $x_1,x_3$ and $v$ commute. Since $v \neq 1$ and $v$ commutes with $u$, we have $C_\F(v) = C_\F(u)$, so $x_1\in C_\F(u)$, hence $g = x_1^kv^{-n_1}\in C_\F(u)$, which contradicts the fact that $g$ and $u$ are non-commuting.
	\end{proof}

	We may now prove \Cref{uniform related ultrapower}.
	
	\begin{proof}[Proof of \Cref{uniform related ultrapower}]  Let $\Omega$ be such a set. By \Cref{related subsets and pk} $\Omega$ is the union of at most $d(s,L)$ subsets $\Omega_i$ for which there is some $k_i \ge 1$ such that $p_{k_i}(\Omega_i)$ is contained in a cyclic coset in $\F$. By \Cref{powers}, either $|\Omega_i|\leq2$ or $k_i=1$ or $p_{k_i}(\Omega_i)$ is contained in an abelian centralizer. In the latter case, by commutative transitivity of $\F$, $\Omega_i$ itself is contained in an abelian centralizer. The result follows. 
	\end{proof}

\section{Proof of \Cref{main theorem}}\label{secproof main}

	We may now complete the proof of \Cref{main theorem}.

   	\begin{proof}
		 Let $\varepsilon \in (0,1)$. Pick the smallest integer $s\geq 2$ such that $b_s:=\frac{2\sqrt{s-1}}{s} <\frac{\varepsilon}{4}$ and let $L = 2M_{1/2,s}$ as in \Cref{small relation if big norm}.  Let $\F$ be a nonabelian free group and $\mu$ a probability measure on $\F$ with $\|\lambda(\mu)\|> {\varepsilon}$. We may find unit vectors $v,w$ in $L^2 \F$ such that $\langle \lambda(\mu) v, w \rangle >\varepsilon$.  Let $\Omega:=\{g \in \F, \langle \lambda(g) v, w \rangle > 2b_s\}$. Note that 
		 $$\varepsilon < \langle \lambda(\mu) v, w \rangle  \leq 2b_s \mu(\F \setminus \Omega) + \mu(\Omega)$$ so $\mu(\Omega)>\varepsilon/2$. For every $g_1,\ldots,g_s \in \Omega$ we have $\|\lambda(\mu_S)\| \ge \langle \lambda(\mu_S)v,w \rangle >2b_s$, where $\mu_S=(\delta_{g_1}+\ldots+\delta_{g_s})/s$. So \Cref{small relation if big norm} implies that $\Omega$ is $(s,L)$-related. By \Cref{uniform related ultrapower}, $\Omega$ is contained in at most $c(s,L)$ cyclic cosets in $\F$. In particular, one of them has $\mu$-mass at least $\mu(\Omega)/c(s,L)\ge \varepsilon/2c(s,L)$. Since $s \leq \lceil (8\varepsilon^{-1})^2 \rceil$, and $L$ depends only on $s$, this gives the desired modulus $\delta=\delta(\varepsilon)>0$.
\end{proof}

We end by including the details of the (routine) proof of \Cref{frf}: 

	\begin{proof} Let $\delta=\delta_\Gamma(\eps)>0$ be a (strictly increasing) modulus for the one-percent property for  $\Gamma$.  Pick a probability measure $\mu$ on $L$ such that $\|\lambda(\mu)\|> \varepsilon$. Let $t=\|\lambda(\mu)\|-\varepsilon$. For each small $\eta>0$, we may find a finitely supported probability measure $\mu_1$ on $L$ with  $\|\mu - \mu_1\|_1 < \eta$. Then $\|\lambda(\mu)-\lambda(\mu_1)\| \leq \eta$, so  $\|\lambda(\mu_1)\|\geq \varepsilon +t - \eta$. Let $S=\Supp(\mu_1)^{\pm 1} \cup \{1\}$ and pick a group homomorphism $\varphi : L\to \Gamma$ that is injective on $S^{8}$  and push forward the measure $\mu_1$ to a probability measure $\nu$ on $\Gamma$. We have $\|\lambda(\nu)\|\ge \|\lambda(\mu_1)\|$. Thus there is an abelian subgroup $A$ of $\Gamma$ and $g \in \Gamma$ such that $\nu(gA)\geq\delta_\Gamma(\varepsilon+t- \eta)$. Write $\Sigma := S\cap \varphi^{-1}(gA)$. Then $\mu(\Sigma) = \mu_1(\Sigma) +( \mu( \Sigma) - \mu_1(\Sigma)) \geq  \mu_1(\Sigma)- \eta\geq \delta_\Gamma(\varepsilon+t- \eta) - \eta$. If $\eta>0$ is chosen small enough, this is $\ge \delta_\Gamma(\varepsilon)$. It remains to show that $\Sigma$ is contained in a coset of an abelian subgroup in $L$. To see this, note that for every $x,y,z,t\in \Sigma$ we have $\varphi([x^{-1}y,z^{-1}t]) = 1$ hence we also have $[x^{-1}y,z^{-1}t] = 1$ as $\varphi$ is injective on $S^8$. Therefore, the subgroup $B$ of $L$ generated by the quotients $x^{-1}y$ for $x,y\in \Sigma$ is abelian, and clearly one has $\Sigma\subset xB$ for any $x\in \Sigma$ as desired.
	\end{proof}

	\section{Further remarks}\label{further}
	
	We include here some further remarks obtained with some AI help.
	
	\begin{rem}[An explicit bound on the one-percent modulus]\label{explicit} By our proof of \Cref{main theorem}, a lower bound for the one-percent modulus of free groups depends only on the modulus of uniform exactness of free groups, namely the quantity $M_{1/2,s}$ from \eqref{ue}. Ozawa's brilliant proof of uniform exactness for free groups makes key use of ultraproducts, and is thus non effective. However, Ozawa \cite{ozawa-personal} shared with the first author on September 24, 2026 a new argument obtained a day earlier in conversation with ChatGPT 6 Astra, which yields the following explicit modulus of uniform exactness: $M_{\delta,s} \leq (4s)^6/\delta^4$. In view of the proof of \Cref{main theorem} and the bound for $c(s,L)$ obtained in \Cref{uniform related ultrapower}, this gives $L=2M_{1/2,s} \leq (8s)^6$ and hence the following explicit lower bound on the one-percent modulus for free groups: $\delta(\varepsilon) \ge 2^{-(8s)^7} \ge \exp(-8^{21} \varepsilon^{-14})$. We suspect that a much better bound ought to hold, perhaps even a polynomial bound in $\varepsilon$. 
	\end{rem}

	\begin{rem}[The ninety-nine percent property]\label{ninety} One may define an analogous \emph{ninety-nine percent property} of a group $\Gamma$, by asking that for every $\varepsilon>0$ there is $\delta>0$ such that every probability measure $\mu$ on $\Gamma$ with $\|\lambda(\mu)\|\ge 1-\delta$ gives $\mu$-mass at least $1-\varepsilon$ to a coset of an amenable subgroup. It was shown in \cite[Theorem 7.6]{becker-breuillard} that $\GL_d(\C)$, viewed as a discrete group, has the ninety-nine percent property. It turns out that the one percent property implies the ninety-nine percent property in any group:
	
	\begin{prop} If $\Gamma$ is a discrete group with the one percent property (with modulus $\delta_1(\varepsilon)>0$), then it has the ninety-nine percent property (with modulus $\delta_{99}(\varepsilon)=\varepsilon \delta_1(1/2)^4/2000$).
	\end{prop}

	\begin{proof}The following argument was suggested by ChatGPT 6 Astra in September 2026. Suppose that $\mu$ is a probability measure on $\Gamma$ such that $\|\lambda(\mu)\|\ge 1-\delta$. Then $\|\lambda(\widetilde{\mu})\|=\|\lambda(\mu)\|^2 \ge (1-\delta)^2\ge 1-2\delta$. Pick an integer $N$ such that $(1-2\delta)^N \ge 1/2$. Then $\|\lambda(\widetilde{\mu}^N)\| \ge 1/2$ and we may apply the one percent property of $\Gamma$ with $\varepsilon=1/2$ and conclude that there is a coset $gA$ of an amenable subgroup $A$ of $\Gamma$ such that $\widetilde{\mu}^N(gA)\ge \delta_0$, where $\delta_0=\delta_1(1/2)>0$ depends only on $\Gamma$. Without loss of generality we may assume that $\delta_0<1/2$. 
	
	Consider now the action of $\Gamma$ on the coset space $\Gamma/A$, and let $\nu$ be the push-forward of $\widetilde{\mu}^N$ under the map $\Gamma \to \Gamma/A$. Then $\nu$ has an atom of mass $\ge \delta_0$. List the $\nu$-mass of atoms in decreasing order $M_1 \geq M_2 \geq \ldots \geq M_n \geq\ldots$. Then $\sum M_i \leq 1$ and $M_1 \ge \delta_0$. It follows that there is $k\leq k_0=\lfloor 3/\delta_0 \rfloor$, such that $M_k-M_{k+1} > c_0:=\delta_0^2/5$. Indeed, if not we would have $1\ge M_1+\ldots +M_{k_0} \ge M_1 + M_1-c_0 + \ldots + M_1-(k_0-1)c_0 = k_0M_1-k_0(k_0-1)c_0/2$. Noting that $k_0\delta_0 \ge 2$ this would give $1 \ge 2-k_0^2c_0/2$, so $k_0^2c_0\ge 2$. However  $k_0^2c_0=(k_0\delta_0)^2/5\leq 9/5<2$. For this $k$ let
	$$E:=\{x\in \Gamma/A, \nu(x) \ge M_k\}$$
	and set $H=\{g \in \Gamma, gE=E\}$. Note that $E$ is finite, and that $H$ contains the pointwise stabiliser of $E$ (which is a subgroup of  a conjugate of $A$) as a subgroup of finite index. In particular $H$ is amenable. Moreover if $g \notin H$, then there is $x \in E$ such that $g^{-1}x \notin E$. In particular $\|\lambda_{\Gamma/A}(g) \nu - \nu\|_2 \ge |\lambda_{\Gamma/A}(g) \nu(x) - \nu(x)|=|\nu(g^{-1}x)-\nu(x)| \ge M_k-M_{k+1} \ge \delta_0^2/5$.
	
	It follows that $$\sum_{g \in \Gamma} \widetilde{\mu}(g)\|\lambda_{\Gamma/A}(g) \nu - \nu\|_2^2 \ge (1-\widetilde{\mu}(H))  \delta_0^4/25.$$
	
	On the other hand a simple spectral argument allows to bound the left hand side as follows: expanding the square we have
	 $$\sum_{g \in \Gamma} \widetilde{\mu}(g)\|\lambda_{\Gamma/A}(g) \nu - \nu\|_2^2 = 2\langle (1-\lambda_{\Gamma/A}(\widetilde \mu))\nu, \nu \rangle_{L^2(\Gamma/A)} = 2\langle P^{2N}(1-P) \delta_A,\delta_A\rangle$$
	 where $P=\lambda_{\Gamma/A}(\widetilde{\mu})$ is a non-negative self-adjoint operator. By the spectral theorem, $$\|P^{2N}(1-P)\|\leq \sup_{t \in [0,1]} t^{2N}(1-t)\leq 1/N.$$ Hence
	 $$\widetilde{\mu}(H) \ge 1- 50/N\delta_0^4.$$
	Choose $N$ maximal such that $(1-2\delta)^N \ge 1/2$. Assuming (as we may) that $\delta<1/8$, this yields $N \ge 1/(40\delta)$. Hence  $\max_{g \in \Gamma} \mu(Hg) \ge \widetilde{\mu}(H) \ge 1- 2000\delta \delta_0^{-4}$. This gives the ninety-nine percent property with modulus $\delta_{99}(\varepsilon)=\varepsilon \delta_0^4/2000$. 
	\end{proof}
\end{rem}

\begin{exem}\label{BS} The  Baumslag-Solitar groups $BS(n,m)=\langle a,t | t^{-1}a^nt=a^m\rangle$ with $\gcd(|n|,|m|)=1$ and $\min\{|n|,|m|\}>1$ are free-by-solvable, yet they do not have the ninety-nine percent property (let alone the one percent property). Indeed, their Cayley graphs with generating set $S_k=\{1,t^{\pm 1}, a^{\pm n^k}\}$ converge when $k$ tends to infinity to the Cayley graph of their metabelian quotient, image of $a \mapsto (x \mapsto x+1)$ and $t \mapsto (x \mapsto \frac{n}{m}x)$ in the affine group of $\Q$. See \cite[\S 3]{osin}, \cite[Prop. 13.3]{arzhantsevaetal}, so $\|\lambda(\mu_k)\| \to 1$, where $\mu_k$ is the symmetric measure on $S_k$.
\end{exem}

\begin{exem}\label{resfree} The infinite direct sum $\Gamma=\bigoplus_{i\ge 1} \F^{(i)}$, where each $\F^{(i)}$ is a nonabelian free group $\F=\langle a,b \rangle$ is an example of a group with the ninety-nine percent property (as follows easily from the fact that two elements of $\Gamma$ either commute or generate a nonabelian free subgroup) but \emph{without the one percent property}. To see this,  take the measure $\mu_n$ which is the uniform average of $n$ copies of the uniform measure $\mu$ on a  long geometric progression $bP$, where $P:=\{a^k, |k|\leq L\}$, so $\mu_n$ is supported on the union of the ``coordinate subgroups'' $\F^{(i)}$, $i\in [\![1,n]\!]$.   Cosets of amenable subgroups of $\Gamma$ are contained in a direct sum of cyclic cosets $C:=\bigoplus_{i \ge 1} C^{(i)}$. If such a coset intersects two distinct coordinate subgroups nontrivially, then it must be the coset of the identity hence a direct sum of cyclic subgroups. However $\mu$ gives mass at most $1/|P|$ to any abelian subgroup of $\F$. Hence $\mu_n(C)\le  \max\{1/n,1/|P|\}$. However, $\|\lambda(\mu_n)\|\ge 1/4$: indeed, take $\xi=(1_{P}+1_{bP})/2 \in L^2(\F)$ and note that $\langle \lambda(\mu)\xi, \xi \rangle \ge \|\xi\|^2/4$ so $\langle \lambda(\mu_n)\eta, \eta \rangle \ge \|\eta\|^2/4$ for $\eta:=\xi\otimes \cdots \otimes \xi$. Note finally that $\Gamma$ is residually free, so \Cref{frfg} cannot be upgraded to residually free groups.
\end{exem}

The following is a standard way to establish the ninety-nine percent property, as done e.g. in \cite[\S 6.3]{becker-breuillard}. It applies to \Cref{resfree} above and \Cref{fgresfree} below.

\begin{prop} Suppose a discrete group $\Gamma$ satisfies the strong Tits alternative: there is $N<\infty$ such that for every finite unital symmetric set $S$ in our group $\Gamma$ either $\langle S \rangle$ is amenable or there is $a,b \in S^N$ that are generators of a nonabelian free group. Then $\Gamma$ has the ninety nine percent property.
\end{prop}

\begin{proof} Let $\mu$ be a probability measure with $\|\lambda(\mu)\|> 1-\varepsilon$. Then there are unit vectors $\xi, \eta$  in $L^2$  such that $\|\lambda(\mu)\xi - \eta\|\leq \varepsilon$. Then $$\sum_g \mu(g)\|\lambda(g)\xi-\eta\|^2=\sum_g \mu(g)(2-2\langle \lambda(g)\xi,\eta\rangle)=2\langle \eta-\lambda(\mu)\xi, \eta \rangle \leq 2\varepsilon$$
Let $E:=\{g \in \Gamma, \|\lambda(g)\xi-\eta\|^2  \le 2M\varepsilon \}$. Then $\mu(E^c) \leq 1/M$. Pick $g_0 \in E$. Then for any $g \in E$ we have $\|\lambda(g_0)\xi - \lambda(g)\xi\| \leq 2\sqrt{2M \varepsilon}$, so for any $a \in S:=g_0^{-1}E$, $\|\lambda(a) \xi - \xi\|^2\leq 8M  \varepsilon$. Hence $\|\lambda(a) \xi - \xi\|\leq N \sqrt{8M  \varepsilon}$ for all $a \in S^N$. It follows $S$ must generate an amenable subgroup as soon as $N \sqrt{8M  \varepsilon}<\kappa_{F_2}$, where $\kappa_{F_2}$ is the Kazhdan constant on $L^2$ for the generators of the nonabelian free group $F_2$.  So  $E$ is contained in an amenable coset to which $\mu$ gives mass at least $1-1/M$.
\end{proof}

\begin{exem}\label{fgresfree} It is possible to convert the \Cref{resfree} into a finitely generated one. The following construction gives a $3$-generated group $G$ which contains $\Gamma=\bigoplus_{i\ge 1} \F^{(i)}$ and has the property that every pair $a,b \in G$ either commutes or generates a nonabelian free group. Hence $G$ fails the one-percent  but enjoys the ninety-nine percent property. Take $\{h_k, k\ge 1\}$ generators of $\Gamma$, consider new free variables  $u,v,x,y$ and form $s_0:=u$ and $s_k:=v^{-k}uh_kv^k$ for $k \ge 1$ and $t_k:=x^{k+1}y^{k+1}x^{k+1}$ for $k \ge 0$. Consider the free products $A:=\Gamma * F(u,v)$ and $B=F(x,y)$ and set $G$ to be the amalgamated product $G:=A *_{s_k=t_k, k \ge 0} B$. Then clearly $G=\langle x,y,v \rangle$. Moreover $\langle s_k, k\ge 0\rangle$ (resp. $\langle t_k, k\ge 0\rangle$) is malnormal in $A$ (resp. $B$); see Theorem 8 and the proof of its corollary in \cite{karrass-solitar}. This places us in the situation of \cite[Theorem 6]{karrass-solitar}, which asserts that every $2$-generated subgroup of such an amalgam $G$ is either nonabelian free, or conjugates into one of the factors. In the latter case, it is either free or conjugates into $\Gamma$, hence is either free or abelian as desired.
\end{exem}

	\bibliography{mybib.bib}
	\bibliographystyle{plain}
\end{document}